\documentclass{amsart}
\usepackage{amsmath,amsfonts,amsthm,amssymb,xspace}
\usepackage{color, pst-all,graphicx}

\numberwithin{equation}{section}

\newtheorem{coro}[equation]{Corollary}
\newtheorem*{conjA}{Conjecture~$\ConjA$}
\newtheorem*{conjAp}{Conjecture~$\ConjAp$}
\newtheorem{lemm}[equation]{Lemma}
\newtheorem{prop}[equation]{Proposition}
\newtheorem*{thrmA}{Theorem~1}

\newtheorem*{prop16}{Proposition~1.6}
\newtheorem*{prop114}{Proposition~1.14}

\theoremstyle{definition}
\newtheorem{defi}[equation]{Definition}
\newtheorem{exam}[equation]{Example}
\newtheorem{rema}[equation]{Remark}

\psset{unit=1mm}
\psset{fillcolor=white}
\psset{dotsep=1.5pt}
\psset{dash=1.5pt 1.5pt}
\psset{linewidth=0.8pt}
\psset{arrowsize=3.5pt}
\psset{doublesep=1pt}
\psset{coilwidth=1.2mm}
\psset{coilheight=2}
\psset{coilaspect=20}
\psset{coilarm=2}
\newpsstyle{double}{linewidth=0.5pt,doubleline=true,arrowsize=5pt}
\newpsstyle{exist}{linestyle=dashed}
\newpsstyle{thin}{linewidth=0.5pt}
\newpsstyle{thinexist}{linewidth=0.5pt, linestyle=dashed}
\newpsstyle{back}{linewidth=3mm,linearc=1}
\definecolor{color1}{rgb}{.88,.88,.88}
\definecolor{color2}{rgb}{1,.85,0.85}
\definecolor{color20}{rgb}{.88,.88,.88}
\definecolor{color21}{rgb}{1,.8,.8}
\definecolor{color22}{rgb}{1,.75,.75}
\definecolor{color3}{rgb}{1,0,0}

\renewcommand\aa{a}
\newcommand\aav{\underline\aa}
\newcommand\act{\mathbin{\scriptscriptstyle\bullet}}
\newcommand\Att{\widetilde{\VR(2.1,0)\smash{\mathrm{A}}}_2}
\newcommand\bb{b}

\newcommand\bbv{\underline\bb}
\newcommand\brk{\mathrm{\scriptscriptstyle b\HS{-0.1}r\HS{-0.1}e\HS{-0.1}a\HS{-0.1}k}}
\newcommand\can{\iota}
\newcommand\cc{c}
\newcommand\ccv{\underline\cc}
\newcommand\cl[1]{[#1]}
\newcommand\clp[1]{[#1]^{\HS{-0.2}\scriptscriptstyle+}}
\newcommand\ConjA{\mathbf{A}}
\newcommand\ConjAp{\smash{\mathbf{A}{\HS{-0.5}\raisebox{4pt}{\mbox{$\mathrm{\scriptscriptstyle p\HS{-0.1}a\HS{-0.1}d\HS{-0.1}d\HS{-0.1}e\HS{-0.1}d}$}}}}}
\newcommand\dd{d}

\renewcommand\dh[1]{\Vert#1\Vert}

\newcommand\dive{\le}

\newcommand\EG[1]{\mathcal{U}(#1)}

\newcommand\equivp{\equiv^{\HS{-0.1}\scriptscriptstyle+}}
\newcommand\ew{\varepsilon}
\newcommand\f[1]{\mathrm{f}[#1]}
\newcommand\ff{f}

\newcommand\FR[1]{\mathcal{F}_{\HS{-0.4}#1}}
\renewcommand\gcd{\wedge}
\newcommand\gcdt{\mathbin{\widetilde\wedge}}
\renewcommand\ge{\geqslant}
\renewcommand\gg{g}
\newcommand\GG{G}
\newcommand\GR[2]{\langle#1\mid#2\rangle}

\newcommand\HS[1]{\hspace{#1ex}}

\newcommand\ii{i}

\newcommand\inv{^{-1}}
\newcommand\INV[1]{\overline{#1}}
\newcounter{ITEM}
\newcommand\ITEM[1]{\setcounter{ITEM}{#1}\leavevmode\hbox{\rm(\roman{ITEM})}}

\newcommand\kk{k}
\renewcommand\l[1]{\mathrm{l}[#1]}
\newcommand\lcm{\vee}
\newcommand\lcmt{\mathbin{\widetilde\lcm}}
\renewcommand\le{\leqslant}

\newcommand\mm{m}
\newcommand\mst{m_{st}} 
\newcommand\MM{M}
\newcommand\MMt{M_0}
\newcommand\MON[2]{\langle#1\mid#2\rangle^{\HS{-0.5}\scriptscriptstyle+}}
\newcommand\nn{n}

\newcommand\NNNN{\mathbb{N}}

\newcommand\One[1]{\underline1_{#1}}
\newcommand\one{\underline1}
\newcommand\pdots{\HS{0.2}{\cdot}{\cdot}{\cdot}\HS{0.2}}
\newcommand\pp{p}

\newcommand\PropH{\mathrm{H}}
\newcommand\qq{q}

\newcommand\rd{\Rightarrow}
\newcommand\RD[1]{\mathcal{R}_{\scriptscriptstyle\HS{-0.3}#1}}
\newcommand\rds{\rd^{\HS{-0.3}*}}
\newcommand\rdsp{\rd_{\HS{-0.4}\spl}}

\newcommand\Red[1]{R_{#1}}
\newcommand\Redbrk[1]{R_{#1}^{\HS{0.2}\brk}}
\newcommand\Redtrim[1]{R^{\mathrm{\HS{0.2}\scriptscriptstyle t\HS{-0.1}r\HS{-0.1}i\HS{-0.1}m}}_{#1}}
\newcommand\resp{\mbox{\it resp}.\ }

\newcommand\rr{r}
\newcommand\RR{R}
\newcommand\RRR{\mathcal{R}}

\newcommand\RRt{\widetilde{R}}
\newcommand\sdots{ / \pdots / }
\newcommand\SD[1]{\mathcal{S}_{\scriptscriptstyle\HS{-0.3}#1}}

\newcommand\spl{\mathrm{\scriptscriptstyle s\HS{-0.1}p\HS{-0.1}l\HS{-0.1}i\HS{-0.1}t}}
\newcommand\spec{\rightsquigarrow}
\renewcommand\ss{s}
\renewcommand\SS{S}
\newcommand\SSb{\overline{S}}

\newcommand\SSS{\mathcal{S}}
\renewcommand\tt{t}
\newcommand\tta{\mathtt{a}}
\newcommand\ttb{\mathtt{b}}
\newcommand\ttc{\mathtt{c}}

\newcommand{\ua}{\underline{a}}
\newcommand{\ub}{\underline{b}}

\newcommand\under{\ensuremath{\mathbin{\backslash}}}
\newcommand\uu{u}

\def\VR(#1,#2){\vrule width0pt height#1mm depth#2mm}

\newcommand\vv{v}

\newcommand\wdots{, ...\HS{0.2},}

\newcommand\wl{\lambda}
\newcommand\ww{w}

\newcommand\xx{x}

\newcommand\yy{y}

\newcommand\zz{z}

\title[Padding and AT monoids of sufficiently large type]{Multifraction reduction IV: Padding and Artin--Tits monoids of sufficiently large type}

\author{Patrick Dehornoy}
\address{P.D., Laboratoire de Math\'ematiques Nicolas Oresme, CNRS UMR 6139, Universit\'e de Caen, 14032 Caen cedex, France, and Institut Universitaire de France}
\email{patrick.dehornoy@unicaen.fr}
\urladdr{www.math.unicaen.fr/\~{}dehornoy}

\author{Derek F. Holt}
\address{D.H., Mathematics Institute, University of Warwick, Coventry CV4 7AL, UK}
\email{D.F.Holt@warwick.ac.uk}

\author{Sarah Rees}
\address{S.R., School of Mathematics and Statistics, University of Newcastle, Newcastle NE3 1ED, UK}
\email{Sarah.Rees@newcastle.ac.uk}



\begin{document}


\maketitle

\begin{abstract}
We investigate the padded version of reduction, an extension of multifraction reduction as defined in \texttt{arXiv:1606.08991}, and connect it both with ordinary reduction and with the so-called Property~$\PropH$. As an application, we show that all Artin--Tits groups of sufficiently large type satisfy some weakening Conjecture~$\ConjAp$ of Conjecture~$\ConjA$, thus showing that the reduction approach is relevant for these groups. 
\end{abstract}

Reduction of multifractions, which was introduced in~\cite{Dit} and~\cite{Diu}, is a new approach to the word problem for Artin-Tits groups and, more generally, for groups that are enveloping groups of monoids in which the divisibility relations have weak lattice properties (``gcd-monoids''). It is based on a rewrite system (``$\RRR$-reduction'') that extends the usual free reduction for free groups, as well as the rewrite systems known for Artin--Tits groups of spherical type, and more generally Garside groups. It was proved in~\cite{Dit} that $\RRR$-reduction is convergent for all Artin--Tits groups of type~FC, and in~\cite{Diu} that a certain condition called semi-convergence, weaker than convergence, is sufficient to obtain the decidability of the word problem, leading to the main conjecture (``Conjecture~$\ConjA$'') that $\RRR$-reduction is semi-convergent for every Artin--Tits monoid.

The aim of the current paper is to exploit the observation that semi-convergence up to Turing-computable padding, a weakening of semi-convergence, is again sufficient to solve the word problem. By {\em padding}, we mean the insertion of an even number of trivial components at the beginning of a multifraction.

The main results we prove are as follows. First, we have a simple criterion for the word problem:

\begin{prop16}
If $\MM$ is a strongly noetherian gcd-monoid with finitely many basic elements, for which $\RRR$-reduction is semi-convergent up to $\ff$-padding for some Turing-computable map~$\ff$, then the word problem of $\EG\MM$ is decidable.
\end{prop16}

Next, we establish a simple connection between the padded version of semi-conver\-gence, the semi-convergence of a variant of $\RRR$-reduction (``split reduction'' or ``$\SSS$-reduction'') and Property~$\PropH$ of~\cite{Dia, Dib, GR}:

\begin{prop114}
If $\MM$ is a gcd-monoid and $(\SS, \RR)$ is an lcm-presentation for~$\MM$, then the following are equivalent:\\
\null\HS3\ITEM1 $\RRR$-reduction is semi-convergent for~$\MM$ up to padding;\\
\null\HS3\ITEM2 $\SSS$-reduction is semi-convergent for~$\MM$;\\
\null\HS3\ITEM3 Property~$\PropH$ is true for~$(\SS, \RR)$.
\end{prop114}

Finally, we consider the specific case of Artin--Tits groups. In view of Proposition~\ref{P:PaddedSC2WP}, we propose

\begin{conjAp}
For every Artin--Tits monoid, $\RRR$-reduction is semi-convergent up to $\ff$-padding for some Turing-computable map~$\ff$.
\end{conjAp}

By the results of~\cite{Dit}, Conjecture~$\ConjAp$ is true for every Artin--Tits monoid of type~FC. Here we prove:

\begin{thrmA}
Conjecture~$\ConjAp$ is true for all Artin--Tits monoids of sufficiently large type.
\end{thrmA}

We recall from~\cite{HR2} that an Artin-Tits group is said to be of {\em sufficiently large type} if, in any triangle in the associated Coxeter diagram, either no edge has label $2$, or all three edges have label $2$, or at least one edge has label $\infty$. The result follows directly from the more precise result stated as Proposition~\ref{P:Main} below, which gives an explicit quadratic upper bound on the padding that is needed. The proof relies on a careful analysis of the techniques of~\cite{HR2} and~\cite{GR}. With this result, the family of Artin--Tits for which multifraction reduction is relevant is greatly enlarged. 

The paper is organised in two sections. The first one is devoted to padded reduction and its variants in a general context of gcd-monoids, and contains a proof of Propositions~\ref{P:PaddedSC2WP} and~\ref{P:H}. The second section is devoted to the specific case of Artin--Tits monoids of sufficiently large type, with a proof of Theorem~1.

\section{Padded multifraction reduction}\label{SSplitRed}

After recalling in Subsection~\ref{SSRed} the definitions that we need for multifraction reduction and the rewrite system~$\RD\MM$, we introduce in Subsection~\ref{SS:PaddedSC} padded versions of semi-convergence and use them to solve the word problem of the enveloping group. Next, we introduce in Subsection~\ref{SSSplitRed} a new rewrite system~$\SD\MM$, a variant of $\RD\MM$ called split reduction, and we connect its semi-convergence with the padded semi-convergence of~$\RD\MM$. Finally, we establish the connection with subword reversing and Property~$\PropH$ in Subsection~\ref{SSPropH}.

\subsection{Multifraction reduction}\label{SSRed}

If $\MM$ is a monoid, we denote by~$\EG\MM$ the enveloping group of~$\MM$, and by~$\can$ the canonical (not necessarily injective) morphism from~$\MM$ to~$\EG\MM$. We say that a finite sequence $\aav = (\aa_1 \wdots \aa_\nn)$ of elements of~$\MM$, also called a \emph{multifraction} on~$\MM$ and denoted by $\aa_1 \sdots \aa_\nn$, \emph{represents} an element~$\gg$ of~$\EG\MM$ if
\begin{equation}\label{E:AltProd}
\smash{\gg = \can(\aa_1) \can(\aa_2)\inv \can(\aa_3) \pdots \can(\aa_\nn)^{(-1)^{\nn - 1}}}
\end{equation} 
holds in~$\EG\MM$. In this context, the parameter~$\nn$ is called the \emph{depth} of~$\aav$, denoted by~$\dh\aav$, and the right hand side of~\eqref{E:AltProd} is denoted by~$\can(\aav)$. We use $\FR\MM$ for the family of all multifractions on~$\MM$, and $\One\nn$ for the depth~$\nn$ multifraction with all entries equal to $1$, skipping~$\nn$ when no ambiguity is possible. Our aim is to recognise which multifractions represent~$1$ in~$\EG\MM$. 

We collect together the definitions that we need concerning monoids.

\begin{defi}
A monoid~$\MM$ is called a \emph{gcd-monoid} if it is cancellative, $1$ is its only invertible element, and any two elements of~$\MM$ admit a left- and a right-gcd, where, for~$\aa, \bb$ in~$\MM$, we say that $\aa$ \emph{left-divides}~$\bb$, written $\aa \dive \bb$, if we have $\bb = \aa\xx$ for some~$\xx$ in~$\MM$, and that $\dd$ is a \emph{left-gcd} of~$\aa$ and~$\bb$ if $\dd$ left-divides~$\aa$ and~$\bb$ and every common left-divisor of~$\aa$ and~$\bb$ left-divides~$\dd$. Right-division and right-gcd are defined symmetrically, with $\bb = \xx\aa$ replacing $\bb = \aa\xx$.

An \emph{atom} in~$\MM$ is an element that is not expressible as $\aa\bb$ with $\aa, \bb \not=1$. A right- (\resp left-) \emph{basic element} is one that is obtained from atoms using the operation~$\under$ (\resp $/$) defined by $\aa (\aa {\under} \bb) = \aa \lcm \bb$ (\resp $(\aa / \bb)\bb = \aa \lcmt \bb$).

The monoid~$\MM$ is called \emph{noetherian} if there is no infinite descending sequence with respect to proper left- or right-division, and \emph{strongly noetherian} if there exists a map $\lambda$ from $\MM \setminus \{1\}$ to the positive integers satisfying $\lambda(\aa\bb) \ge \lambda(\aa) + \lambda(\bb)$. 
\end{defi}

Throughout the paper, we restrict our attention to gcd-monoids: Artin--Tits monoids are typical examples. In a gcd-monoid, any two elements~$\aa, \bb$ that admit a common right-multiple(\resp left-multiple) admit a least one, the \emph{right-lcm} (\resp \emph{left-lcm}) of~$\aa$ and~$\bb$, denoted by~$\aa \lcm \bb$ (\resp $\aa \lcmt \bb$).

If $\MM$ is a gcd-monoid, a family~$\RD\MM$ of rewrite rules on~$\FR\MM$ is defined in~\cite{Dit}: for $\aav, \bbv$ in~$\FR\MM$, and for $\ii \ge 1$ and $\xx \in \MM$, we write $\aav \act \Red{\ii, \xx} = \bbv$ if we have $\dh\bbv = \dh\aav$, $\bb_\kk = \aa_\kk$ for $\kk \not= \ii - 1, \ii, \ii + 1$, and there exists~$\xx'$ satisfying
$$\begin{array}{lccc}
\text{for $\ii$ even:}
&\bb_{\ii-1} = \aa_{\ii-1} \xx', 
&\xx \bb_\ii = \aa_\ii \xx' = \xx \lcm \aa_\ii, 
&\xx \bb_{\ii+1} = \aa_{\ii+1},\\
\text{for $\ii \ge 3$ odd:\qquad}
&\bb_{\ii-1} = \xx' \aa_{\ii-1}, 
&\bb_\ii \xx = \xx' \aa_\ii = \xx \lcmt \aa_\ii, 
&\bb_{\ii+1} \xx = \aa_{\ii+1},\\
\text{for $\ii = 1$:}
&&\bb_\ii \xx = \aa_\ii, 
&\bb_{\ii+1} \xx = \aa_{\ii+1}.
\end{array}$$
We say that $\aav$ \emph{reduces} to~$\bbv$ in one step, and write $\aav \rd \bbv$, if $\aav \act \Red{\ii, \xx} = \bbv$ holds for some~$\ii$ and some $\xx \not= 1$. We use $\rds$ for the reflexive--transitive closure of~$\rd$. 

By~\cite[Lemma~3.8 and Cor.~3.20]{Dit}, $\aav \rds \bbv$ implies that $\aav$ and~$\bbv$ represent the same element in~$\EG\MM$, and, conversely, the relation~$\can(\aav) = \can(\bbv)$ is essentially the equivalence relation generated by~$\rd$ (up to deleting trivial final entries). Furthermore, whenever the monoid~$\MM$ is noetherian, $\RRR$-reduction is terminating for~$\MM$. 

By~\cite[Prop.~4.16]{Dit} and~\cite[Prop.~3.2]{Diu}, $\RRR$-reduction is locally confluent for~$\MM$ if and only if $\MM$ satisfies the \emph{$3$-Ore condition}, meaning that any three elements of~$\MM$ that pairwise admit a common right- (\resp left-) multiple admit a global one. In this case, $\RRR$-reduction is convergent, implying that a multifraction~$\aav$ represents~$1$ in~$\EG\MM$ if and only if $\aav \rds \one$ is satisfied. In the case of Artin--Tits monoids, this exactly corresponds to type~FC.

When $\MM$ does not satisfy the $3$-Ore condition, $\RRR$-reduction is not confluent for~$\MM$, and no convergence result can be expected. However, a condition weaker than convergence can be used to solve the word problem:

\begin{defi}[\cite{Diu}]
We say that $\RRR$-reduction is \emph{semi-convergent for~$\MM$} if
$\aav \rds\nobreak \one$ holds for every multifraction~$\aav$ that
represents~$1$ in~$\EG\MM$. 
\end{defi} 

Convergence implies semi-convergence, but explicit examples of gcd-monoids~$\MM$ for which $\RRR$-reduction is semi-convergent but not convergent are known~\cite{Div}. However:

\begin{prop}[{\cite[Prop.~3.16]{Diu}}]\label{P:WP}
 If $\MM$ is a strongly noetherian gcd-monoid with finitely many basic elements for which $\RRR$-reduction is semi-convergent, then the word problem of~$\EG\MM$ is decidable.
\end{prop}

\begin{conjA}[\cite{Diu}]
$\RRR$-reduction is semi-convergent for every Artin--Tits monoid.
\end{conjA}

Conjecture~$\ConjA$ is true for every Artin--Tits monoid of FC~type, but other cases remain open. Every Artin--Tits monoid satisfies the finiteness assumptions of Proposition~\ref{P:WP} and, therefore, the word problem of the enveloping group of every Artin--Tits monoid satisfying Conjecture~$\ConjA$ is decidable.

\subsection{Padded semi-convergence}\label{SS:PaddedSC}

We now introduce a weakening of semi-conver\-gence, in which trivial initial entries may be added to multifractions. For $\pp$ a nonnegative integer and $\aav$ a multifraction, we write $\One{2\pp} / \aav$ for the multifraction obtained from~$\aav$ by adding $2\pp$~trivial entries on the left of~$\aav$. Note that the number of~$1$s that are added must be even in order to preserve the image of the multifraction in the group.

\begin{defi}\label{D:InflSC}
If $\ff$ is a map from~$\FR\MM$ to~$\NNNN$, we say that $\RRR$-reduction is semi-convergent for~$\MM$ \emph{up to $\ff$-padding} if $\One{2\ff(\aav)} / \aav \rds\nobreak \one$ holds for every
multifraction~$\aav$ that represents~$1$ in~$\EG\MM$. 
\end{defi}

By definition, semi-convergence is semi-convergence up to $0$-padding. Because $\aav \rds\nobreak \bbv$ trivially implies $\One{2\pp} /\aav \rds \One{2\pp} / \bbv$ for every~$\pp$, semi-convergence of $\RRR$-reduction up to $\ff$-padding implies its semi-convergence up to $\ff'$-padding for~$\ff' \ge \ff$ (meaning $\forall \aav {\in} \FR\MM\,(\ff'(\aav) \ge \ff(\aav))$): appending trivial initial entries preserves the existing reductions and possibly adds new ones.

\begin{prop}\label{P:PaddedSC2WP}
If $\MM$ is a strongly noetherian gcd-monoid with finitely many basic elements for which $\RRR$-reduction is semi-convergent up to $\ff$-padding for some Turing-computable map~$\ff$ on~$\FR\MM$, then the word problem of~$\EG\MM$ is decidable.
\end{prop}

Speaking of a Turing-computable map defined on~$\FR\MM$ makes sense because, under the assumptions of Proposition~\ref{P:PaddedSC2WP}, $\MM$ admits a finite presentation in terms of the finite atom family~$\SS$ of~$\MM$, and every element of the monoid~$\MM$ is represented by finitely many words in~$\SS$ only, which implies the decidability of the word problem of~$\MM$ with respect to~$\SS$. Then a Turing-computable map on~$\MM$ means one that is induced by a Turing-computable map on words in~$\SS$.

If $(\SS, \RR)$ is a monoid presentation, that is, $\RR$ is a family of pairs of words in~$\SS$, we denote by~$\equivp_\RR$, or simply~$\equivp$, the congruence on~$\SS^*$ generated by~$\RR$, implying $\MON\SS\RR = \SS^*{/}{\equivp}$. For~$\ww$ a word in~$\SS$, we denote by~$\clp\ww$ the element of~$\MON\SS\RR$ represented by~$\ww$, that is, the $\equivp$-class of~$\ww$. Let~$\SSb$ be a disjoint copy of~$\SS$ that consists of a new element~$\INV\ss$ for each element~$\ss$ of~$\SS$. For~$\ww$ a word in $\SS \cup \SSb$, we denote by~$\INV\ww$ the word obtained from~$\ww$ by exchanging~$\ss$ and~$\INV\ss$ everywhere and reversing the order of letters. In this context, the element of~$\SS$ (\resp $\SSb$) are called positive (\resp negative), and we denote by~$\cl\ww$ the element of the group~$\GR\SS\RR$ represented by~$\ww$ when $\INV\ss$ represents~$\ss\inv$.

\begin{proof}[Proof of Proposition~\ref{P:PaddedSC2WP}]
The argument is the same as the one in~\cite{Diu} for Proposition~\ref{P:WP}. Let $\SS$ be the atom set of~$\MM$, and let $\ww$ be a word in $\SS \cup \SSb$. To decide whether $\ww$ represents~$1$ in~$\EG\MM$, we first express it as $\ww_1 \ww_2\inv \ww_3 \pdots$ with $\ww_1, \ww_2$, ... in~$\SS^*$ (no negative letters). Then, writing $\underline\ww = (\ww_1, \ww_2, ...)$ and $\clp{\underline\ww}$ for $\clp{\ww_1} / \clp{\ww_2} / \pdots$, we append~2$\ff(\clp{\underline\ww})$ trivial entries on the left, and exhaustively enumerate all sequences~$\underline{\ww}'$ satisfying $\One{2\ff(\clp{\underline\ww})} / \clp{\underline{\ww}} \rds \clp{\underline{\ww}'}$: this is possible because, under the assumptions of the statement, the existence of common multiples and, from there, the relation~$\rds$, are decidable. Then $\ww$ represents~$1$ in~$\EG\MM$ if and only if a trivial sequence $\ew \sdots \ew$ appears in the list, where $\ew$ is the empty word.
\end{proof}

We are thus led to the new conjecture, Conjecture~$\ConjAp$, stated in the introduction. Conjecture~$\ConjAp$ is \textit{a priori} weaker than Conjecture~$\ConjA$ but, by Proposition~\ref{P:PaddedSC2WP}, it nevertheless implies the decidability of the word problem of every Artin--Tits group for which it is true.

\subsection{Split reduction}\label{SSSplitRed}

We now connect the padded version of semi-conver\-gence of $\RRR$-reduction with the ordinary semi-convergence of a new rewrite system, a variant of $\RRR$-reduction that can naturally be called \emph{split reduction}. 

In the definition of multifraction reduction used in \cite{Dit,Diu,Div} and repeated in Section~\ref{SS:PaddedSC}, in order for $\aav \act \Red{\ii, \xx}$ to be defined, $\xx$ and~$\aa_\ii$ have to admit a common multiple, a left- or right-multiple depending on the parity of~$\ii$. In order to define a variant of this reduction, we now relax the assumption and only require that $\xx$ and some divisor of~$\aa_\ii$ admit a common multiple; then, we again remove~$\xx$ from~$\aa_{\ii + 1}$ and push it to the left or right side using the lcm operation, this time through the divisor of~$\aa_\ii$ that is involved and not necessarily through the whole of~$\aa_\ii$.

\begin{defi}
(See Figure~\ref{F:Sharp} right.) Assume that $\MM$ is a gcd-monoid. For $\aav, \bbv$ in~$\FR\MM$, and for $\ii \ge 1$ even (\resp odd) and $\xx, \yy$ in~$\MM$ such that $\xx \lcm \yy$ (\resp $\xx \lcmt \yy$) exists, we write $\aav \act \Redbrk{\ii, \xx, \yy} = \bbv$ if we have $\dh\bbv = \dh\aav + 2$, $\bb_\kk = \aa_\kk$ for $\kk < \ii$, $\bb_\kk = \aa_{\kk - 2}$ for $\kk > \ii + 3$, and
$$\begin{array}{lccc}
\text{for $\ii$ even:}
&\yy \bb_\ii= \aa_\ii, 
&\yy \bb_{\ii + 1} = \xx \bb_{\ii + 2} = \xx \lcm \yy, 
&\xx \bb_{\ii+3} = \aa_{\ii+1},\\
\text{for $\ii$ odd:\qquad}
&\bb_\ii \yy= \aa_\ii, 
&\bb_{\ii + 1} \yy = \bb_{\ii + 2} \xx = \xx \lcmt \yy, 
&\bb_{\ii+3} \xx = \aa_{\ii+1}.
\end{array}$$
In addition we define ``trimming'' rules $\Redtrim\ii$, such that $\aav \act \Redtrim\ii = \bbv$ holds if we have $\aa_{\ii + 1} = 1$, $\dh\bbv = \dh\aav - 2$, and $\bb_\kk = \aa_\kk$ for $\kk < \ii$, $\bb_\kk = \aa_{\kk + 2}$ for $\kk > \ii + 1$, and $\bb_\ii = \aa_\ii \aa_{\ii + 2}$ (\resp $\aa_{\ii + 2} \aa_\ii$) if $\ii$ is odd (\resp even). 
We write $\aav \rdsp \bbv$ if either $\aav \act \Redbrk{\ii,\xx,\yy} = \bbv$ holds for some~$\ii, \xx, \yy$ with $\xx\yy \not= 1$ or $\aav \act \Redtrim\ii = \bbv$ holds for some $\ii$, and denote by $\rdsp^*$ the reflexive--transitive closure of~$\rdsp$. The rewrite system~$\SD\MM$ so obtained is called \emph{split reduction} on~$\MM$.
\end{defi}

\begin{figure}[htb]
\begin{center}
\begin{picture}(50,29)(0,1)
\psset{nodesep=0.7mm}
\psline[style=back,linecolor=color2]{c-c}(0,0.5)(9.5,0.5)(9.5,20)(39.5,20)(39.5,30.5)(50,30.5)
\psline[style=back,linecolor=color1]{c-c}(0,0)(10,0)(10,10)(40,10)(40,30)(50,30)
\put(-5,0){...}
\pcline{->}(0,0)(10,0)
\pcline{<-}(10,0)(10,10)\trput{$\aa_{\ii - 1}$}
\pcline{<-}(10,10)(10,20)\trput{$\xx'$}
\psline[style=thin](9,1)(8,1)(8,20)(9,20)\put(1,10){$\bb_{\ii - 1}$}
\pcline{->}(10,10)(40,10)\tbput{$\aa_\ii$}
\pcline{->}(10,20)(40,20)\taput{$\bb_\ii$}
\psarc[style=thin](10,20){3}{270}{360}
\pcline[linewidth=1.5pt,linecolor=color3,arrowsize=1.5mm]{<-}(40,10)(40,20)\tlput{$\xx$}
\pcline{<-}(40,20)(40,30)\tlput{$\bb_{\ii + 1}$}
\pcline{->}(40,30)(50,30)
\put(52,30){...}
\psline[style=thin](41,10)(42,10)(42,29)(41,29)\put(43,20){$\aa_{\ii + 1}$}
\put(23,14){$\Leftarrow$}
\end{picture}
\hspace{5mm}
\begin{picture}(50,29)(0,1)
\psset{nodesep=0.7mm}
\psline[style=back,linecolor=color2]{c-c}(0,0.5)(9.5,0.5)(9.5,10.5)(25,10.5)(25,20)(39.5,20)(39.5,30.5)(50,30.5)
\psline[style=back,linecolor=color1]{c-c}(0,0)(10,0)(10,10)(40,10)(40,30)(50,30)
\put(-5,0){...}
\pcline{->}(0,0)(10,0)
\pcline{<-}(10,0)(10,10)\trput{$\aa_{\ii - 1}$}\tlput{$\bb_{\ii - 1}$}
\pcline{<-}(25,10)(25,20)\put(18,15.5){$\bb_{\ii + 1}$}
\pcline{->}(10,10)(25,10)\put(15,12){$\bb_\ii$}
\pcline[linewidth=1.5pt,linecolor=color3,arrowsize=1.5mm]{->}(25,10)(40,10)\put(31.5,11.5){$\yy$}
\psline[style=thin](11,9)(11,8)(40,8)(40,9)\put(23,5.5){$\aa_\ii$}
\pcline{->}(25,20)(40,20)\taput{$\bb_{\ii+2}$}
\psarc[style=thin](25,20){3}{270}{360}
\pcline[linewidth=1.5pt,linecolor=color3,arrowsize=1.5mm]{<-}(40,10)(40,20)\tlput{$\xx$}
\pcline{<-}(40,20)(40,30)\put(33,25.5){$\bb_{\ii + 3}$}
\pcline{->}(40,30)(50,30)
\put(52,30){...}
\psline[style=thin](41,10)(42,10)(42,29)(41,29)\put(43,20){$\aa_{\ii + 1}$}
\put(31,15){$\Leftarrow$}
\put(28,14){$\mathrm{\scriptscriptstyle part}$}
\end{picture}
\end{center}
\caption{\small Comparing the reduction~$\Red{\ii, \xx}$ (left) and the split reduction~$\Redbrk{\ii, \xx, \yy}$ (right), here for $\ii$ odd: in both cases, we divide~$\aa_{\ii + 1}$ by~$\xx$ and push~$\xx$ to the left using an lcm, but, for~$\Red{\ii, \xx}$, we demand that $\xx$ crosses the whole of~$\aa_\ii$, whereas, for~$\Redbrk{\ii, \xx, \yy}$, we only require that $\xx$ crosses the possibly proper divisor~$\yy$ of~$\aa_\ii$; as a consequence, we cannot gather the elements at level~$\ii - 1$ and the depth increases by~$2$.}
\label{F:Sharp}
\end{figure}
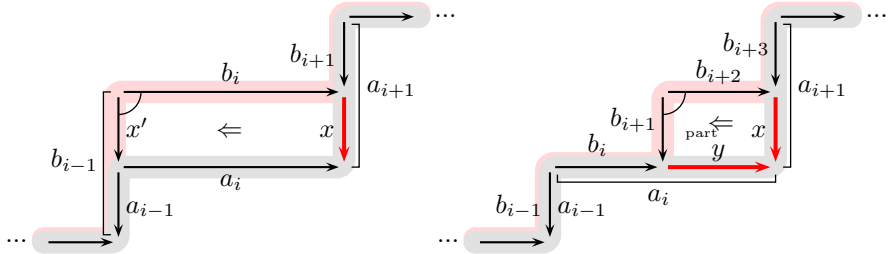

Adding the trimming rules~$\Redtrim\ii$ is a technical artefact that will make statements simpler. It directly follows from the definition that, if $\aav \rdsp^* \bbv$ is true, then $\aav$ and~$\bbv$ represent the same element in the group~$\EG\MM$. A major difference between $\RRR$- and $\SSS$-reductions is that, even if the monoid~$\MM$ is noetherian, no termination can be expected for the latter in general.

\begin{exam}
Let $\MM$ be the Artin--Tits monoid of type~$\Att$, that is, 
$$\MON{\tta, \ttb, \ttc}{\tta\ttb\tta = \ttb\tta\ttb, \ttb\ttc\ttb = \ttc\ttb\ttc, \ttc\tta\ttc = \tta\ttc\tta},$$
and let $\aav = \tta\ttb / \ttc$. Then we find $\aav \act \Redbrk{1, \ttc, \ttb} \Redbrk{1, \ttc, \tta} = 1 / \tta\ttc / \ttc\tta / \ttb / \ttc\ttb / 1$, whence
$$\aav \act \Redbrk{1, \ttc, \ttb} \Redbrk{1, \ttc, \tta} \Redbrk{3, \ttb, \tta} \Redbrk{3, \ttb, \ttc} \Redbrk{5, \tta, \ttc} \Redbrk{5, \tta, \ttb} = 1 / \tta\ttc / 1 / \ttc\ttb / 1 / \ttb\tta / \tta\ttb / \ttc/ \tta\ttc / 1 / \ttb\tta / 1 / \ttc\ttb / 1,$$
that is, with obvious notation, $\aav \rdsp^6 1 / \tta\ttc / 1 / \ttc\ttb / 1 / \ttb\tta / \aav/ \tta\ttc / 1 / \ttb\tta / 1 / \ttc\ttb / 1$, leading to an infinite sequence of split reductions from~$\aav$.
\end{exam}

The following result directly follows from the definition; we state it both and as a useful example for practice and for future reference. 

\begin{lemm} \label{lem:move_to_left}
If $\MM$ is a gcd-monoid, and $\aav = \pdots/ \aa / 1 / \bb \cc / \pdots $ is a multifraction on~$\MM$ with elements $\aa$, $1$, and~$\bb \cc $ in positions $i-1$, $i$, and $i+1$, then we have
$$\begin{array}{cccl}
\aav \act \Red{\ii , \bb} &=\pdots / \aa / 1 / \bb\cc / \pdots\act \Red{\ii , \bb} &=\pdots/\aa \bb /1/\cc /\pdots &\text{for $i$ even,}\\
\aav \act \Red{\ii , \cc} &=\pdots / \aa / 1 / \bb\cc / \pdots\act \Red{\ii , \cc} &=\pdots /\cc \aa /1/\bb /\pdots &\text{for $i$ odd}.
\end{array}$$
In particular, for $\aa =1$, we find
$$\begin{array}{cccl}
\aav \act \Red{\ii , \bb} &=\pdots / 1 / 1 / \bb\cc / \pdots\act \Red{\ii , \bb} &=\pdots/ \bb /1/\cc /\pdots &\text{for $i$ even,}\\
\aav \act \Red{\ii , \cc} &=\pdots / 1 / 1 / \bb\cc / \pdots\act \Red{\ii , \cc} &=\pdots / \cc /1/\bb /\pdots &\text{for $i$ odd}.
\end{array}$$
\end{lemm}

\begin{proof}
The diagrams provide a proof.
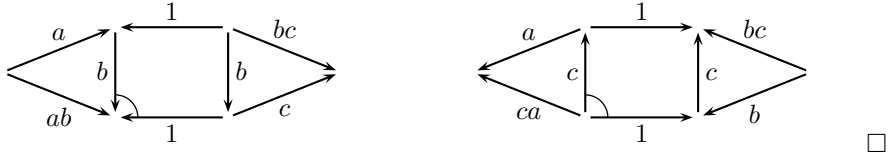
\begin{center}
\begin{picture}(45,7)(0,7)
\psset{nodesep=0.7mm}
\pcline{->}(0,6)(15,0)\tbput{$\aa\bb $}
\pcline{<-}(15,0)(30,0)\tbput{$1$}
\pcline{->}(30,0)(45,6)\tbput{$\cc $}
\pcline{->}(0,6)(15,12)\taput{$\aa$}
\pcline{<-}(15,12)(30,12)\taput{$1$}
\pcline{->}(30,12)(45,6)\taput{$\bb \cc $}
\pcline{->}(15,12)(15,0)\tlput{$\bb $}
\pcline{->}(30,12)(30,0)\trput{$\bb $}
\psarc[style=thin](15,0){3}{0}{90}
\end{picture}
\HS{10}
\begin{picture}(45,7)(0,7)
\psset{nodesep=0.7mm}
\pcline{<-}(0,6)(15,0)\tbput{$\cc\aa$}
\pcline{->}(15,0)(30,0)\tbput{$1$}
\pcline{<-}(30,0)(45,6)\tbput{$\bb $}
\pcline{<-}(0,6)(15,12)\taput{$\aa$}
\pcline{->}(15,12)(30,12)\taput{$1$}
\pcline{<-}(30,12)(45,6)\taput{$\bb \cc $}
\pcline{<-}(15,12)(15,0)\tlput{$\cc $}
\pcline{<-}(30,12)(30,0)\trput{$\cc $}
\psarc[style=thin](15,0){3}{0}{90}
\end{picture}
\vspace{4mm}
\end{center}
\end{proof}

\begin{lemm}\label{L:Part}
Assume that $\MM$ is a gcd-monoid and $\aav, \bbv$ belong to~$\FR\MM$. \\
\null\HS{3}\ITEM1 If $\aav \rds \bbv$ holds, then so does $\aav \rdsp^* \bbv$.\\
\null\HS{3}\ITEM2 If $\aav \rdsp^* \bbv$ holds, then $\One{2\pp} / \aav \rds \bbv / \One{2\qq}$ holds for some~$\pp$ and~$\qq$.
\end{lemm}

\begin{proof}
\ITEM1 For an induction it is sufficient to prove that $\aav \rd \bbv$ implies $\aav \rdsp^* \bbv$. So assume that $\aav \act \Red{\ii, \xx} = \bbv$ with, say, $\ii$~even. Writing $\xx \lcm \aa_\ii = \xx \bb_\ii = \aa_\ii \xx'$ and applying the definition of~$\Redbrk{\ii, \xx, \yy}$, we find
$$\aav \act \Redbrk{\ii, \xx, \aa_\ii} = \aa_1 \sdots \aa_{\ii - 1} /1 / \xx' / \bb_\ii \sdots \bb_\nn = \bb_1 \sdots \bb_{\ii - 2} / \aa_{\ii - 1} /1 / \xx' / \bb_\ii \sdots \bb_\nn.$$
The definition of~$\Red{\ii, \xx}$ gives $\bb_{\ii - 1} = \aa_{\ii - 1}\xx'$, so we deduce
$$\aav \act \Redbrk{\ii, \xx, \aa_\ii} \Redtrim\ii = \bb_1 \sdots \bb_{\ii - 2} / \aa_{\ii - 1}\xx' / \bb_\ii \sdots \bb_\nn = \bbv,$$
whence $\aav \rdsp^2 \bbv$. The argument is similar for $\ii$~odd.

\ITEM2 We prove using induction on~$\mm$ that $\aav \rdsp^\mm \bbv$ implies the existence of~$\pp$ and~$\qq$ satisfying $\One{2\pp} / \aav \rds \bbv / \One{2\qq}$. Put $\nn := \dh\aav$. 
The result is straightforward for~$\mm = 0$ with $\pp = \qq = 0$. 
Suppose that $\mm = 1$. Suppose first $\aav \act \Redbrk{\ii, \xx, \yy} = \bbv$ with, say, $\ii$~even. 
By hypothesis, $\yy$ left-divides~$\aa_\ii$, and we have $\aa_\ii = \yy \bb_\ii$. 
Starting from $\One2 /\aav$, we can first use $\ii$~reduction steps to split~$\aa_\ii$ into~$\bb_\ii$ and~$\yy$ by pushing $\aa_1 \wdots \aa_{\ii - 1}$, and~$\bb_\ii$ to the left using Lemma~\ref{lem:move_to_left} repeatedly: precisely, we find
$$\One2 / \aav \act \Red{2, \aa_1}\Red{3, \aa_2} \pdots \Red{\ii, \aa_{\ii - 1}} \Red{\ii + 1, \bb_\ii}= \aa_1 \sdots \aa_{\ii - 1} / \bb_\ii / 1 / \yy / \aa_{\ii + 1} \sdots \aa_\nn,$$
whence $\One2 /\aav \rd^\ii \bb_1 \sdots \bb_\ii / 1 / \yy / \aa_{\ii + 1} / \bb_{\ii + 4} \sdots \bb_{\nn + 2}$. Let $\aav'$ be the latter multifraction. By hypothesis, $\xx$ divides~$\aa_{\ii + 1}$, which is~$\aa'_{\ii + 3}$, and the lcm of~$\xx$ and~$\yy$, that is, of~$\xx$ and~$\aa_{\ii + 2}$, exists, so $\aav'$ is eligible for~$\Red{\ii + 2, \xx}$. Then we obtain
$$\aav' \act \Red{\ii + 2, \xx} = \bb_1 \sdots \bb_\ii / \bb_{\ii + 1} / \bb_{\ii + 2} / \bb_{\ii + 3} / \bb_{\ii + 4} \sdots \bb_{\nn + 2}\bbv,$$
whence $\One2 /\aav \rds \aav' \rd \bbv$, the expected result with $\pp = 1$ and~$\qq = 0$. 

Suppose now $\aav \act \Redtrim\ii = \bbv$, still with $\ii$~even. By hypothesis, we have $\aa_{\ii + 1} = 1$, and we find
$$\aav \act \Red{\ii + 1, \aa_{\ii + 2}} = \aa_1 \sdots \aa_{\ii - 1} / \aa_\ii\aa_{\ii + 2} /1 / 1 / \aa_{\ii + 3} \sdots \aa_\nn$$
and from there, by the same argument as above with $\One2 / \aav$, we can push $\aa_{\ii + 3} \wdots \aa_\nn$ to the left, obtaining
$$\aav \act \Red{\ii + 1, \aa_{\ii + 2}} \Red{\ii + 2, \aa_{\ii + 3}} \sdots \Red{\nn - 1, \aa_\nn} = \aa_1 \sdots \aa_{\ii - 1} / \aa_\ii\aa_{\ii + 2} /\aa_{\ii + 3} \sdots \aa_\nn / 1 / 1,$$
which implies $\aav \rds \bbv/ \One2$, the expected result with $\pp = 0$ and~$\qq = 1$. The verifications for $\ii$~odd are similar.

Suppose finally $\mm \ge 2$. Write $\aav \rdsp^{\mm - 1} \ccv \rdsp \bbv$. The induction hypothesis implies the existence of~$\pp$ and~$\qq$ satisfying $\One{2\pp} / \aav \rds \ccv / \One{2\qq}$, 
and that of~$\pp'$ and~$\qq'$ satisfying $\One{2\pp'} / \ccv \rds \bbv / \One{2\qq'}$.
As $\RRR$-reduction is compatible with left and right multiplication, $\One{2\pp} / \aav \rds \ccv$ implies $\One{2\pp + 2\pp'} / \aav \rds \One{2\pp'} / \ccv / \One{2\qq}$, and $\One{2\pp'} / \ccv \rds \bbv / \One{2\qq'}$ implies $\One{2\pp'} / \ccv / \One{2\qq} \rds \bbv / \One{2\qq + 2\qq'}$. By transitivity of~$\rds$, we deduce $\One{2\pp + 2\pp'} / \aav \rds \bbv / \One{2\qq + 2\qq'}$ and this is the expected result.
\end{proof}

We deduce a simple connection between $\RRR$- and $\SSS$-reductions to~$\one$:

\begin{prop}\label{P:Part}
If $\MM$ is a gcd-monoid then, for every multifraction~$\aav$ on~$\MM$,
\begin{equation}
\aav \rdsp^* \one \text{\quad is equivalent to \quad} \exists\pp\,(\One{2\pp}/ \aav \rds \one).
\end{equation}
\end{prop}

\begin{proof}
Assume $\aav \rdsp^* \one$, say $\aav \rdsp^* \One\rr$. By Lemma~\ref{L:Part}, there exist~$\pp, \qq$ satisfying $\One{2\pp} / \aav \rds \One{\rr + 2\qq}$, whence $\One{2\pp}/\aav \rds \one$. 

Conversely, assume $\One{2\pp}/ \aav \rds \one$. Suppose first $\dh\aav \le 1$. Then $\One{2\pp}/ \aav \rds \one$ is possible only for $\aav = 1$, in which case $\aav \rdsp^* \one$ is trivial. Now suppose $\dh\aav \ge 2$ with $\aav$ not trivial, that is, at least one entry of~$\aav$ is not~$1$. Let $\ii$ be minimal satisfying $\aa_\ii \not= 1$. Then $\aav$ is eligible for~$\Redbrk{\ii, 1, \aa_1}$, and we find $\aav \act \Redbrk{\ii, 1, \aa_\ii} = \One2 / \aav$, whence $\aav \rdsp^* \One{2\pp} / \aav$ in $\pp$~steps. Next, by Lemma~\ref{L:Part} again, $\One{2\pp}/ \aav \rds \one$ implies $\One{2\pp} / \aav \rdsp^* \one$. We deduce $\aav \rdsp^* \one$ by transitivity of~$\rdsp^*$.
\end{proof}

We naturally say that $\SSS$-reduction is semi-convergent for~$\MM$ if $\aav \rdsp^*\nobreak \one$ holds for every multifraction~$\aav$ that represents~$1$ in~$\EG\MM$. Proposition~\ref{P:Part} implies:

\begin{coro}\label{C:Part}
If $\MM$ is a gcd-monoid, then $\SSS$-reduction is semi-convergent for~$\MM$ if and only if $\RRR$-reduction is semi-convergent for~$\MM$ up to padding.
\end{coro}

\subsection{Connection with subword reversing and Property~$\PropH$}\label{SSPropH}

In the absence of a
a solution of the word problem for general Artin--Tits groups, Property~$\PropH$ was proposed~\cite{Dia} as a weaker statement in the direction of such a solution along the line of Dehn's algorithm for hyperbolic groups, and it was proved that Property~$\PropH$ is indeed satisfied for the standard presentation of all Artin--Tits groups of type~FC~\cite{Dib} and of sufficiently large type~\cite{GR}. We now establish a simple connection between Property~$\PropH$ and multifraction reduction:

\begin{prop}\label{P:H}
If $\MM$ is a gcd-monoid and $(\SS, \RR)$ is an lcm-presentation for~$\MM$ (on both sides), then the following are equivalent: \vspace{0.5mm}\\
\null\HS3\ITEM1 $\RRR$-reduction is semi-convergent for~$\MM$ up to padding;\\
\null\HS3\ITEM2 $\SSS$-reduction is semi-convergent for~$\MM$;\\
\null\HS3\ITEM3 Property~$\PropH$ is true for~$(\SS, \RR)$.
\end{prop}

By Corollary~\ref{C:Part}, \ITEM1 and~\ITEM2 in Proposition~\ref{P:H} are equivalent, and we are left with establishing the equivalence of~\ITEM2 and~\ITEM3, which will occupy the rest of the section. 

We first recall a few definitions. If $(\SS, \RR)$ is a monoid presentation, and $\ww, \ww'$ are words in $\SS \cup \SSb$, then we say that $\ww'$ is obtained from~$\ww$ by one step of \emph{right $\RR$-reversing} if $\ww'$ is obtained by deleting some length~$2$ factor $\INV\ss \ss$ of~$\ww$ or replacing some length~$2$ factor~$\INV\ss \tt$ of~$\ww$ with~$\vv \INV\uu$ such that $\ss\vv = \tt\uu$ is a relation of~$\RR$. Symmetrically, we say that $\ww'$ is obtained from~$\ww$ by one step of \emph{left $\RR$-reversing} if $\ww'$ is obtained by deleting some length~$2$ factor $\ss \INV\ss$ of~$\ww$ or replacing some length~$2$ factor~$\ss \INV\tt$ of~$\ww$ with~$\INV\vv \uu$ such that $\vv\ss = \uu\tt$ is a relation of~$\RR$. 

\begin{defi}\label{D:PropH}
Suppose that $(\SS, \RR)$ is a monoid presentation. If $\ww, \ww'$ are words in $\SS \cup \SSb$, we say that $\ww \spec \ww'$ is true if one can go from~$\ww$ to~$\ww'$ using finitely many \emph{special} transformations of the following types: \ITEM1 replacing a positive factor~$\uu$ of~$\ww$ by~$\uu'$ with $\uu' \equivp \uu$ (``positive equivalence''), \ITEM2 replacing a negative factor~$\INV\uu$ of~$\ww$ by~$\INV{\uu'}$ with $\uu' \equivp \uu$ (``negative equivalence''), \ITEM3 right $\RR$-reversing, \ITEM4 left $\RR$-reversing. We say that $(\SS, \RR)$ satisfies \emph{Property~$\PropH$} if $\cl\ww = 1$ implies $\ww \spec \ew$, where $\ew$ is the empty word.
\end{defi}

All special transformations map a word in $\SS \cup \SSb$ to one that represents the same element in~$\EG\MM$, so $\ww \spec \ew$ always implies $\cl\ww = 1$: Property~$\PropH$ says that this implication is an equivalence. Note that special transformations add no new trivial factor~$\ss\INV\ss$ or~$\INV\ss\ss$, a situation reminiscent of Dehn's algorithm for hyperbolic groups~\cite[Sec.\,1.2]{Dib}.

Let us say that a presentation~$(\SS, \RR)$ of a monoid~$\MM$ is a \emph{right-lcm presentation} if $\RR$ contains one relation for each pair~$(\ss, \tt)$ in~$\SS \times \SS$ such that $\ss$ and~$\tt$ admit a common right-multiple and this relation has the form $\ss\vv = \tt\uu$ where both $\ss\vv$ and $\tt\vv$ represent the right-lcm $\ss \lcm \tt$. A left-lcm presentation is defined symmetrically. By~\cite[Thm.~4.1]{Dfx}, every strongly noetherian gcd-monoid admits a right-lcm presentation and a left-lcm presentation. For instance, the standard presentation of an Artin-Tits monoid is an lcm presentation on both sides. The point here is that, if $(\SS, \RR)$ is a right-lcm presentation, then right-reversing computes the right-lcm in~$\MON\SS\RR$ in the following sense:

\begin{lemm}[{\cite[Prop.~3.6]{Dia}}]\label{L:Lcm}
Assume that $\MM$ is a strongly noetherian gcd-monoid and $(\SS, \RR)$ is a right-lcm presentation for~$\MM$. Then, for all words~$\uu, \vv$  in~$\SS^*$, the elements~$\clp\uu$ and~$\clp\vv$ of~$\MM$ admit a common right-multiple if and only if the right $\RR$-reversing of~$\INV\uu \vv$ leads to  word~$\vv' \INV{\uu'}$ with~$\uu', \vv'$ in~$\SS^*$ and, in this case, one has $\clp{\uu \vv'} = \clp{\vv \uu'} = \clp\uu \lcm \clp\vv$.
\end{lemm}

Of course, if $(\SS, \RR)$ is a left-lcm presentation of~$\MM$, left $\RR$-reversing computes the left-lcm in~$\MM$ in a similar sense.

Our aim from now on is to connect the relation~$\spec$ on words in $\SS \cup \SSb$ with partial reduction of multifractions. The convenient bridge between words and multifractions is as follows:

\begin{defi}\label{def:represents}
Suppose that $\MM = \MON\SS\RR$. A word~$\ww$ in $\SS \cup \SSb$ is said to \emph{represent} a depth~$\nn$ multifraction~$\aav$ on~$\MM$ if there exist words $\ww_1 \wdots \ww_\nn$ in~$\SS^*$ satisfying $\ww = \ww_1 \INV{\ww_2} \ww_3 \INV{\ww_4} \pdots$ and $\clp{\ww_\ii} = \aa_\ii$ for every~$\ii$. We say that $\ww$ \emph{sharply} represents~$\aav$ if, in addition, $\ww_\ii$ is nonempty for $1 < \ii < \nn$.
\end{defi}

A word~$\ww$ in $\SS \cup \SSb$ may represent several multifractions, since the expression of~$\ww$ as $\ww_1 \INV{\ww_2} \ww_3 \INV{\ww_4} \pdots$ need not be unique if no maximality condition is required, but it sharply represents a unique multifraction.

\begin{lemm}\label{L:Connection1}
Assume that $\MM$ is a gcd-monoid with lcm presentation~$(\SS, \RR)$ and $\ww$ is a word in $\SS \cup \SSb$ representing a multifraction~$\aav$. Then $\aav \rdsp \bbv$ implies $\ww \spec \ww'$ for some word~$\ww'$ in $\SS \cup \SSb$ representing~$\bbv$.
\end{lemm}

\begin{proof}
Put $\nn = \dh\aav$, and let $\ww_1 \wdots \ww_\nn$ satisfy $\ww = \ww_1 \INV{\ww_2} \ww_3 \INV{\ww_4} \pdots$ and $\clp{\ww_\ii} = \aa_\ii$ for every~$\ii$. Suppose first $\aav \act \Redtrim{\ii} = \bbv$ with, say, $\ii$~odd. By hypothesis, we have $\aa_\ii = 1$, which requires $\ww_\ii = \ew$, because there is no nontrivial invertible element in~$\MM$, so the empty word must be the only element of its $\equivp$-class. Then we have $\bb_\ii = \aa_\ii \aa_{\ii + 2}$, and $\ww$ also represents~$\bbv$, as witnessed by the decomposition $\ww_1 \wdots \ww_{\ii - 1}, \ww_\ii\ww_{\ii + 2}, \ww_{\ii + 3} \wdots \ww_\nn$. The argument is similar for $\ii$ even, replacing $\ww_\ii\ww_{\ii + 2}$ with $\ww_{\ii + 2}\ww_\ii$.

Suppose now $\aav \act \Redbrk{\ii, \xx, \yy} = \bbv$ with, say, $\ii$~even. The assumption that $\aav \act \Redbrk{\ii, \xx, \yy}$ is defined implies that $\yy$ left-divides~$\aa_\ii$, and $\xx$ left-divides~$\aa_{\ii + 1}$. It follows that there exist words~$\uu, \vv, \uu'', \vv''$ in~$\SS$ satisfying
$$\ww_\ii \equivp \uu\uu'' \text{\quad and \quad} \ww_{\ii + 1} \equivp \vv\vv'' \text{\quad with \quad} \clp\uu = \yy \text{\quad and \quad}\clp\vv = \xx.$$
By Lemma~\ref{L:Lcm}, the assumption that $\xx \lcm \yy$ exists implies that the right-$\RR$-reversing of~$\INV\uu \vv$ leads to some -word~$\vv' \INV{\uu'}$ with~$\uu', \vv'$ in~$\SS^*$ such that $\uu\vv'$ and $\vv\uu'$ both represent~$\xx \lcm \yy$. Put $\ww'_\kk := \ww_\kk$ for $\kk < \ii$, $\ww'_\ii := \uu''$, $\ww'_{\ii + 1} := \vv'$, $\ww'_{\ii + 2} := \uu'$, $\ww'_{\ii + 3} := \vv''$, and $\ww'_{\kk + 2} := \ww_\kk$ for $\kk > \ii + 1$, and $\ww':= \ww'_1 \INV{\ww'_2} \ww_3 \INV{\ww_4} \pdots$. By construction, $\ww'$ is obtained from~$\ww$ using special transformations (namely positive and negative equivalences and right-reversing) and, by the definition of~$\Redbrk{\ii, \xx, \yy}$, it represents~$\bbv$. The argument is similar for $\ii$ odd, exchanging right- and left-reversings.
\end{proof}

Some care is needed for the other direction, because positive and negative equivalences might cause trouble. This is where sharp representations are useful.

\begin{lemm}\label{L:Connection2}
Assume that $\MM$ is a gcd-monoid with lcm presentation~$(\SS, \RR)$ and $\ww$ is a word in $\SS \cup \SSb$ sharply representing a multifraction~$\aav$. Then $\ww \spec \ww'$ implies $\aav \rdsp^* \bbv$ for some multifraction~$\bbv$ sharply represented by~$\ww'$.
\end{lemm}

\begin{proof}
Put $\nn = \dh\aav$, and let $\ww_1 \wdots \ww_\nn$ satisfy $\ww = \ww_1 \INV{\ww_2} \ww_3 \INV{\ww_4} \pdots$ and $\clp{\ww_\ii} = \aa_\ii$ for every~$\ii$.
Suppose first that $\ww'$ is obtained from~$\ww$ by positive equivalence. Because the representation is sharp, the transformation involves one word~$\ww_\ii$, with odd~$\ii$ only. Then we have $\ww'_\ii \equivp \ww_\ii$, and $\ww'_\kk = \ww_\kk$ for~$\kk \not= \ii$, so $\ww'$ sharply represents~$\aav$, and the result is true with $\bbv = \aav$. The case when $\ww'$ is obtained from~$\ww$ by negative equivalence is similar.

Suppose now that $\ww'$ is obtained from~$\ww$ by one step of right-reversing. This means that there exists an even integer~$\ii$, generators~$\ss, \tt$ in~$\SS$, words~$\uu', \vv'$ in~$\SS^*$, and a relation $\ss\vv = \tt \uu$ in~$\RR$ satisfying $\ww_\ii = \tt \vv'$, $\ww_{\ii + 1} = \ss \uu'$, and such that $\ww'$ is obtained from~$\ww$ by replacing the factor $\INV{\ww_\ii} \ww_{\ii + 1}$, which is $\INV{\vv'} \, \INV\tt \ss \uu'$, with $\INV{\vv'} \uu \INV{\vv} \uu'$. Now, if we define $\ww'_\kk := \ww_\kk$ for $\kk < \ii$, $\ww'_\ii := \vv'$, $\ww'_{\ii + 1} := \uu$, $\ww'_{\ii + 2} := \vv$, $\ww'_{\ii + 3} := \uu'$, and $\ww'_{\kk + 2} := \ww_\kk$ for $\kk > \ii + 1$, then we have $\ww' = \ww'_1 \INV{\ww'_2} \ww'_3 \INV{\ww'_4} \pdots$, whereas, by very definition, the words $\ww'_1 \wdots \ww'_{\nn + 2}$ witness that $\ww'$ represents the multifraction $\aav \act \Redbrk{\ii, \ss, \tt}$. If all words~$\uu, \vv, \uu', \vv'$ are nonempty, then the representation is sharp, and taking $\bbv := \aav \act \Redbrk{\ii, \ss, \tt}$ gives the result. Otherwise, we restore sharpness by inserting trimming steps. Assume for instance that $\uu$ is empty, but $\uu'$ and $\vv\vv'$ are not. Then $\aav \act \Redbrk{\ii, \ss, \tt}$ is eligible for~$\Redtrim\ii$, and $\ww'$ sharply represents $\bbv := \aav \act \Redbrk{\ii, \ss, \tt}\Redtrim\ii$. The other cases are treated similarly, possibly using several trimmings. Note that free reduction, which corresponds to~$\ss = \tt$ and $\uu = \vv = \ew$, is not special. Finally, the argument for left-reversing is similar, with $\ii$ now necessarily odd.
\end{proof}

We can now complete the argument.

\begin{proof}[Proof of Proposition~\ref{P:H}]
As said above, it only remains to establish the equivalence of~\ITEM2 and~\ITEM3. Suppose first that $\SD\MM$ is semi-convergent, and let $\ww$ be a word in $\SS \cup \SSb$ representing~$1$ in~$\EG\MM$. Let $\ww_1 \wdots \ww_\nn$ be words in~$\SS^*$ satisfying $\ww = \ww_1 \INV{\ww_2} \ww_3 \pdots$, and let $\aav:= \clp{\ww_1} \sdots \clp{\ww_\nn}$. Then we find 
\begin{equation}\label{E:Conn}
\can(\aav) = \can(\clp{\ww_1}) \can(\clp{\ww_2})\inv \can(\clp{\ww_3}) \pdots = \cl{\ww_1} \cl{\ww_2}\inv \cl{\ww_3} \pdots = \cl{\ww},
\end{equation}
so $\aav$ represents~$1$ in~$\EG\MM$. By hypothesis, this implies $\aav \rdsp \one$. By Lemma~\ref{L:Connection1}, we must have $\ww \spec \ww'$ for some word~$\ww'$ in $\SS \cup \SSb$ representing~$\one$. The word~$\ww'$ must be the empty word, so we have $\ww \spec \ew$. Thus \ITEM2 implies~\ITEM3.

Suppose now that Property~$\PropH$ is satisfied for~$(\SS, \RR)$, and let $\aav$ be a multifraction satisfying $\can(\aav) = 1$. First, by applying as many trim steps~$\Redtrim\ii$ as possible to~$\aav$, we obtain $\aav \rdsp^* \aav'$ for some multifraction~$\aav'$ satisfying $\aa'_\ii \not= 1$ for $1 < \ii < \dh{\aav'}$. Now let $\ww$ be a word in $\SS \cup \SSb$ representing~$\aav'$. The conditions $\aa'_\ii \not= 1$ imply that $\ww$ sharply represents~$\aav'$. By construction, $\can(\aav') = 1$ holds in~$\EG\MM$, and, therefore, by the computation of~\eqref{E:Conn}, $\ww$ represents~$1$ in~$\EG\MM$. By Property~$\PropH$, we must have $\ww \spec \ew$. By Lemma~\ref{L:Connection2}, there exists a multifraction~$\bbv$ sharply represented by~$\ew$. The latter is necessarily of the form~$\one$, implying $\aav' \rdsp \bbv$, whence $\aav \rdsp \one$. Thus \ITEM3 implies~\ITEM2.
\end{proof}

It was noted in~\cite{Dia} that Property~$\PropH$ need not imply the decidability of the word problem in general; Proposition~\ref{P:H} precisely measures the gap between the two properties: by Proposition~\ref{P:PaddedSC2WP}, the decidability of the word problem follows from the semi-convergence of $\RRR$-reduction up to a Turing-computable padding, whereas Property~$\PropH$ entails no effective upper bound on the padding that is needed.

\begin{rema}
As stated above, Property~$\PropH$ makes sense only for those gcd-monoids that admit a presentation that is both a right-lcm and a left-lcm presenta\-tion---as is the case for the standard presentation of Artin--Tits groups. If $\MM$ is a strongly noetherian gcd-monoid with atom set~$\SS$, then $\MM$ admits a right-lcm presentation~$(\SS, \RR)$ and a left-lcm presentation~$(\SS, \RRt)$ which need not coincide. Then Property~$\PropH$ can be stated for the double presentation~$(\SS, \RR, \RRt)$, using~$\RR$ for right-reversing and~$\RRt$ for left-reversing, and all results stated above extend without change to this general version.
\end{rema}

\section{Artin--Tits groups of sufficiently large type}
\label{sec:sufflarge}

We now investigate a special family of gcd-monoids, namely Artin--Tits monoids of sufficiently large type, for which we establish the padded semi-convergence result stated as Theorem~1.

A presentation~$(\SS, \RR)$ of a monoid~$\MM$ is an \emph{Artin--Tits presentation} if, for each pair~$(\ss, \tt)$ in~$\SS \times \SS$ with $\ss \neq \tt$, the set~$\RR$ contains at most one relation, and that relation has the form $\ss\tt\pdots = \tt\ss\pdots$, relating two alternating products of length $\mst$, for some integer $\mst \geq 2$. If $\mst=2$, then $(\ss,\tt)$ is called a \emph{commuting pair}, and if there is no such relation between $\ss$ and $\tt$, then $(\ss,\tt)$ is called  a \emph{free pair}, and we define $\mst=\infty$. The associated \emph{Coxeter diagram} is a complete graph with vertex set  $\SS$, and a label $\mst$ on the edge joining $\ss$ and $\tt$. A monoid $\MM$ is called an Artin--Tits monoid if it possesses an  Artin--Tits presentation $(\SS,\RR)$. We say that $\MM$ and its enveloping group $\EG\MM$ are of {\em sufficiently large type} if, in any triangle in the associated Coxeter diagram, either no edge has label $2$, or all three edges have label $2$, or at least one edge has label $\infty$.

Since all relations in $\RR$ have the form $\uu = \vv$, where $\uu$ and~$\vv$ are positive words of the same length, we see that the word-length map~$\wl$ induces a well-defined map on the presented monoid, verifying strong noetherianity. We use $\wl$ also to denote the induced map on the monoid and, for $\aa \in \MM$, we call $\wl(\aa)$ the \emph{word-length} of~$\aa$. We extend~$\wl$ to multifractions by~$\wl(\aav) := \sum_\ii \wl(\aa_\ii)$, and call~$\wl(\aav)$ the \emph{word-length} of~$\aav$. 

\subsection{The result}\label{sec:result}

The main result that we shall prove is:

\begin{prop}\label{P:Main}
If $\MM$ is an Artin--Tits monoid of sufficiently large type and $\aav$ is a multifraction of word-length~$\ell$ that represents~$1$ in~$\EG\MM$, then we have $\One{2\pp} /\aav \rds \one$ with $\pp = 3 \ell (\ell+2)/4$.
\end{prop}

Note that, since all of the defining relators of the Artin-Tits groups have even length, a word in $\SS \cup \SSb$ that represents the identity must have even length, so $\ell$ must be even in the above proposition.

So, $\RRR$-reduction is semi-convergent for every Artin--Tits monoid of sufficiently large type up to a quadratic padding. As the square of word-length is obviously Turing-computable, Proposition~\ref{P:Main} verifies that $\MM$ satisfies Conjecture~$\ConjAp$, and therefore that the word problem of the associated Artin--Tits group is decidable (as was known already). By Proposition~\ref{P:H}, we also obtain an alternative proof to the main result of~\cite{GR}, namely that the standard presentation of Artin-Tits groups of sufficiently large type satisfies Property~$\PropH$.

Proposition~\ref{P:Main} follows from directly the following technical result.
We say that a multifraction~$\aav$ is \emph{geodesic} if its word-length is minimal over all multifractions that represent the same group element.

\begin{prop}\label{P:MainProp}
If $\MM$ is an Artin--Tits monoid of sufficiently large type and $\aav$ is a multifraction of word-length~$\ell$ that is not geodesic, then $\One{6\ell}/\ua$ is $\RRR$-reducible to a multifraction of word-length $\ell -2$. 
\end{prop}

The proof makes use of the existence of an effective rewrite system for  Artin--Tits groups of sufficiently large type that is defined and explored in \cite{HR1,HR2}, most of whose rules involve two generators only.  We show that  appropriate sequences of rules from that system can be simulated within multifraction reduction to reduce  $\One{6\ell}/\ua$ as required.

The following subsection, Subsection~\ref{sec:strategy}, explains our strategy. Further technical results are established in Subsection~\ref{sec:technical}, and then the proof given in Subsection~\ref{sec:proof}.

\subsection{Strategy for the proof of Proposition~\ref{P:MainProp}}\label{sec:strategy}

For the rest of this section, we assume that $\MM$ is an Artin--Tits monoid of sufficiently large type, and
that $(\SS,\RR)$ is the associated presentation.

We shall refer to the elements of $\SS$ as
{\em generators} and to those of $\SS \cup \SSb$ as {\em letters}.
For $\ss \in \SS$, we call $\ss$ the {\em name} of both of
the letters $\ss$ and $\INV\ss$. We use $\f{\ww}$ and $\l{\ww}$ to denote the
first and last letter of a word $\ww$ in $\SS \cup \SSb$.
Furthermore, in situations where an element~$\aa$ of~$\MM$ is represented
by a unique word $\ww$ in $\SS$, we use $\f{\aa}$ and $\l{\aa}$ also to
denote the first and last letters of $\ww$. 

The rewrite system described in \cite{HR1,HR2} is based on two kinds of length 
preserving rules relating words, $\tau$-moves and $\kappa$-moves, together with
free reduction. By definition, a $\tau$-move rewrites a certain type of 2-generator subword
$w$ on non-commuting non-free generators to another word $\tau(w)$ on the same
two generators, where $w$ and $\tau(w)$ have the same length and represent
the same group element. A $\kappa$-move applies commutation relations to
subwords on two or more generators. In this article, in order to simplify the
exposition, we choose to break $\kappa$-moves down as sequences
of $\kappa$-moves on 2-generator subwords, and we introduce the term
$\tau^+$-move to mean either a $\tau$-move or
a $\kappa$-move applied to a 2-generator subword.

So $\tau^+$-moves can be applied only to certain geodesic 2-generator words,
which we call {\em critical} words. When $v$ is critical, then $\tau^+(v)$ is
also critical, on the same set of generators. 
In all cases, $\f{v}$ and $\f{\tau^+(v)}$ have distinct
names, as do $\l{v}$ and $\l{\tau^+(v)}$. Furthermore, if $v$ is neither a
positive nor a negative word, then exactly one of $\f{v}$ and $\f{\tau^+(v)}$
lies in $\SS$, and the same is true for $\l{v}$ and $\l{\tau^+(v)}$.

Suppose that $w$ is a freely reduced non-geodesic word in~$\SS \cup \SSb$. Then, by \cite[Proposition 3.2]{HR2}, $w$ admits a so-called {\em leftward length reducing critical sequence}. This consists of a certain type of factorisation (see \cite{HR2} for a precise definition) $w = \alpha w_1 w_2 \pdots w_k \beta$ for some $k \ge 1$, where $w_1 \wdots w_k$ are geodesic $2$-generator subwords of $w$, and $w_k$ is critical. We can locate such a decomposition, if it exists, by a finite search.

To perform the length reduction, as described in \cite{HR2}, we first replace $w_k$ by $\tau^+(w_k)$. Then, if $k>1$, we replace $w'_{k-1}$ by $\tau^+(w'_{k-1})$, where $w'_{k-1} := w_{k-1}\f{\tau^+(w_k)}$. We continue in this manner, finally replacing $w'_1$ by $\tau^+(w'_1)$, where $w'_1 := w_1\f{\tau^+(w_2)}$. At that stage, we have $\l{\alpha} = \INV{\f{\tau^+(w'_1)}}$, and we can freely cancel these two adjacent letters.

To mimic the length reduction summarised above, we need to simulate the replacement of $w_k$ by $\tau^+(w_k)$. If $w_k$ is a positive or a negative word then, by definition, $\tau^+(w_k)$ is another such word representing the same group element, with $\f{\tau^+(w_k)} \neq \f{w_k}$. In that case, the only multifraction operations that we may need to carry out are moves that ensure that the part of the multifraction that is represented by $w_k$ occupies a single component.

When $w_k$ is neither a positive nor a negative word, the simulation of the replacement of $w_k$ by $\tau^+(w_k)$ is generally carried out in several stages. Initially, a certain part of the multifraction, which may be spread across several components, is represented by the word $w_k$.  We use $\RRR$-reduction (if necessary) to adjust that part of the multifraction so that it occupies exactly two adjacent components, numbered $i$ and $i+1$ for some $i$, and represents a word $v_k$ with $\cl{w}=\cl{v_k}$. Furthermore, if $\f{w_k} \in \SS$ (i.e. $w_k$ starts with a positive letter), then $i$ is odd and $v_k$ is a concatenation of nonempty positive and negative words, whereas if $\f{w_k} \in \SSb$, then $i$ is even and $v_k$ is a concatenation of nonempty negative and positive words. We then carry out a single $\RRR$-reduction, after which the part of the multifraction in question occupies components $i-1$ and $i-2$, and represents a word $u'_k$ with $\cl{u'_k}=\cl{v_k}$, which is a concatenation of negative and positive words in the first case, and of positive and negative words in the second. 

Let $u_k$ be the maximal proper suffix of $u_k'$. We shall show later in Lemma~\ref{lem:sameletter} that $\f{u_k'}= \f{\tau^+(w_k)}$, so $u_k$ represents the same group element as $\f{\tau^+(w_k)}^{-1}\tau^+(w_k)$. Then, for $k>1$, we adjoin $\f{u_k'}$ to $w_{k-1}$ to give the word $w_{k-1}'$ defined above.

We repeat the above procedure on $w_{k-1}'$ to define words $u'_{k-1}$, $u_{k-1}$ and then, provided that $k>2$, to define $w'_{k-2}$, $u'_{k-2}$, $u_{k-2}$, ..., $w'_1$, $u'_1$, $u_1$. Then the first letter $\f{u'_1}$ of $u'_1$ freely reduces with the final letter $\l{\alpha}$ of $\alpha$, and this cancellation can be effected by a single $\RRR$-reduction.

So the word has effectively been transformed to $\alpha'u_1 \pdots u_k \beta$, where $\alpha'$ is the maximal prefix of~$\alpha$. Now we use $\RRR$-reduction to simulate the replacement of each $u_i$ by a $2$-generator geodesic $v_i$. The resulting word $v = \alpha' v_1 \pdots v_k \beta$, where $|\alpha'| = |\alpha| - 1$, $|v_k| = |w_k|-1$ and, for $1 \le i < k$, $|v_i| = |w_i|$, is the same as the word resulting from the application of sequence of $\tau^+$-moves to $w$ followed by the free cancellation. So we have successfully used $\RRR$-reduction to simulate the reduction resulting from the leftward critical sequence, and reduce the length of the word by~2.

\subsection{Technical results for two-generator Artin--Tits monoids}\label{sec:technical}

In order to simulate the $\tau^+$ transformations, we need a few preliminary results. By definition, the elements involved in the transformations belong to a non-free two-generator submonoid of the ambient monoid, and we establish here various technical statements involving the elements of such submonoids and their enveloping groups. 

The first point is that every element in the group admits distinguished fractional decompositions.

\begin{prop}\label{prop:length2}
If $\MMt$ is a non-free two-generator Artin--Tits monoid, then every element of~$\EG\MMt$ admits a unique expression $\aa \bb\inv$ with $\aa, \bb$ in~$\MMt$ and $\aa \gcdt \bb = 1$, and a unique expression $\cc\inv \dd$ with $\cc, \dd$ in~$\MMt$ and $\cc \gcd \dd = 1$.
\end{prop}

\begin{proof}
The lcm~$\Delta$ of the two generators is a Garside element in~$\MMt$ and,
as a consequence, any two elements of~$\MMt$ admit a common right-multiple
and a common left-multiple. By the classical Ore's theorem, this implies that
every element of~$\EG\MMt$ can be expressed as a right fraction~$\aa\bb\inv$
and as a left fraction~$\cc\inv \dd$. As any two elements of~$\MMt$ admit a
left gcd and a right gcd, dividing the numerator and the denominator of a
right (\resp left) fraction by their right gcd (\resp left gcd) yields
irreducible fractions. Uniqueness follows from the fact that, if the right
gcd of~$\aa$ and~$\bb$ is trivial, then $\cc\aa$ is the right lcm
of~$\cc$ and~$\dd$ whenever we have $\cc \aa = \dd \bb$
(see for instance \cite[Lemma~9.3.5]{ECHLPT}).
\end{proof}

In the above situation, we shall call $\aa\bb^{-1}$ and $\cc^{-1}\dd$ the
\emph{right} and \emph{left} {\em fractional normal forms} of the associated
element.

The next point is that, under mild assumptions, the numerator and the denominator of these fractional normal forms are represented by unique words in the generators:

\begin{lemm}\label{cor:uniquewd}
If $\MMt$ is a non-free two-generator Artin--Tits monoid, and $\aa\bb\inv$ and $\cc\inv\dd$ are the fractional normal forms of an element~$\gg$ of~$\EG\MMt$ and, moreover, $\aa$ and $\bb$ (or $\cc$ and $\dd$) are both $\not= 1$, then none of $\aa$, $\bb$, $\cc$ or $\dd$ is divisible by the Garside element $\Delta$, and there are unique positive words in $\MMt$ representing each of them.
\end{lemm}

\begin{proof}
If, for example, $\aa$ were divisible by $\Delta$ then it would have a representative word with the same last letter as a representative word for~$\bb$, which would violate the condition that $\aa$ and~$\bb$ have no common right divisor. The uniqueness of the word representing $\aa$ is proved in~\cite[Proposition~4.3]{MM}.
\end{proof}

In this situation, we shall unambiguously write $\f{\aa}$ and $\l{\aa}$ for the first and last letters of the unique 2-generator word that represents $\aa$, and similarly for $\bb$, $\cc$, and~$\dd$.

Next, the generators occurring at the ends of the numerators and denominators obey some constraints:

\begin{lemm} \label{lem:letterchange}
If $\MMt$ is a non-free two-generator Artin--Tits monoid, and $\aa\bb\inv$ and $\cc\inv\dd$ are the fractional normal forms of an element~$\gg$ of~$\EG\MMt$ and, moreover, $\aa$ and $\bb$ (or $\cc$ and $\dd$) are both $\not= 1$, then we have $\f{a} \ne \l{c}$ and $\f{b} \ne \l{d}$.
\end{lemm}

\begin{proof} 
We proved in Lemma~\ref{cor:uniquewd} that none of $a,b,c,d$ is divisible by the Garside element $\Delta$ of $\MMt$. Now, from the equality $ca = db$ and the fact that $c$ and $d$ have no nontrivial common left divisor, we have $\f{c} \ne \f{d}$, and so the element $ca$ of~$\MMt$ is represented by more than one word in the generators. By~\cite[Proposition~4.3]{MM} again, any such word must admit $\Delta$ as a subword. Since neither $c$ nor $a$ is divisible by $\Delta$, we have $\Delta = c'a'$, where $c'$ and $a'$ are nontrivial left and right divisors of $c$ and $a$. But the two words representing $\Delta$ consist of alternating letters, and so we have $\f{a} \ne \l{c}$ and similarly $\f{b} \ne \l{d}$.
\end{proof}

Using the previous result, we obtain some control about the first (or the last) generator occurring in the fractional normal form of an arbitrary geodesic word.

\begin{lemm} \label{lem:sameletter}
Assume that $\MMt$ is a non-free two-generator Artin--Tits monoid, and $\aa\bb\inv$ and $\cc\inv\dd$ are the fractional normal forms of an element~$\gg$ of~$\EG\MMt$ lying neither in~$\MMt$ nor in~$\MMt\inv$. Then, if $w$ is a geodesic word representing~$\gg$, we have $\f{a} = \f\ww$ if $\f\ww$ is positive, and $\l\cc = \INV{\f\ww}$ if $\f\ww$ is negative.
\end{lemm}

\begin{proof}
The proof is by induction on the length $n$ of $w$. In this proof, for a generator $x \in S$, we shall write $x^{\pm}$ to mean
either $x$ or $\INV\xx$. Suppose that $w= x_{i_1}^{\pm} \pdots x_{i_n}^{\pm}.$ There is nothing to prove
for $n \le 1$. Let us assume that that $\f{w} =x_{i_1}$, that is, $\f\ww$ is
positive; the other case, when $\f{w}$ is negative, is similar.
Let $v$ be the maximal prefix
$x_{i_1}x_{i_2}^{\pm} \pdots x_{i_{n-1}}^{\pm}$ of $w$.
If $v$ is a positive word, then $ab^{-1}$ with $a = \cl{v}$ and
$b=\cl{x_{i_n}}$ is a fractional normal form of $\cl{w}$, so $\f{a} = x_{i_1}$.
Otherwise, let $a'b'^{-1}$ be the right fractional normal form of $\cl{v}$.
Then by induction we have $\f{a'} = \f{v}=\f{w}$.

Suppose first $\l{w} = \INV\xx_{i_n}$. Now $\cl{a'}$ and $\cl{x_{i_n}b'}$
may have a nontrivial greatest common right divisor $\aa''$ but,
since $w$ is not a negative element, this cannot be the whole of~$a'$.
By Lemma~\ref{cor:uniquewd}, $a'$ is represented by a unique positive word.
Hence $a'=a \aa''$ and so $\f{a}=\f{a'}=\f{w}$.

Otherwise we have $\l{w} = x_{i_n}$. If $x_{i_n}=\f{b'}$
then $x_{i_n}$ simply cancels
with $\l{b'^{-1}}$ and we are done.
Otherwise the left fractional normal form of $\cl{b'^{-1}x_{i_n}}$ is
$\cl{b'}^{-1}\cl{x_{i_n}}$ and, since $\l{a'} \ne \l{b'}$, it follows from Lemma
\ref{lem:letterchange} that the right fractional normal form $a''b''^{-1}$ of
$\cl{b'^{-1}x_{i_n}}$ has $\f{a''} \ne \l{b'}$ and hence $\f{a''} = \l{a'}$.
So $a'a''$ is not divisible by $\Delta$.
Hence $a'a''$ has a unique geodesic representative and so, if $a'a''$ and
$b''$ had a nontrivial common right divisor, then so would $a''$ and $b''$,
which is not the case. 
So $a'a''b''^{-1}$ is the right fractional normal form of $\cl{w}$,
with $\f{a'}=\f{w}$, and we are done.
\end{proof}

We now establish a series of results involving reduction of multifractions, still in the non-free two-generator case. Although very easy, the following result will be crucial.

\begin{lemm}\label{lem:swap_forms}
If $\MMt$ is a non-free two-generator Artin--Tits monoid, and $\aa\bb\inv$ and $\cc\inv\dd$ are the fractional normal forms of an element of~$\EG\MMt$, then we have
$$\begin{array}{cccc}
\pdots /\ss /c/d\tt /\pdots & \rd^{R_{i,d}} &\pdots/\ss a/b/\tt /\pdots &
\text{when $i$ is even}\\
\pdots /\ss /a/\tt b/\pdots & \rd^{R_{i,b}} & \pdots /c\ss /d/\tt /\pdots &
\text{when $i$ is odd},
\end{array}$$
where the multifractions are of depth at least 3, and the entries shown are in
positions $i-1,i,i+1$.In particular, for $\ss =1$, we have
$$\begin{array}{cccc}
\pdots /1/c/d\tt /\pdots & \rd^{R_{i,d}} &\pdots/a/b/\tt /\pdots &
\text{when $i$ is even}\\
\pdots /1/a/\tt b/\pdots & \rd^{R_{i,b}} & \pdots /c/d/\tt /\pdots &
\text{when $i$ is odd},
\end{array}$$
\end{lemm}

\begin{proof}
The diagrams provide a proof.
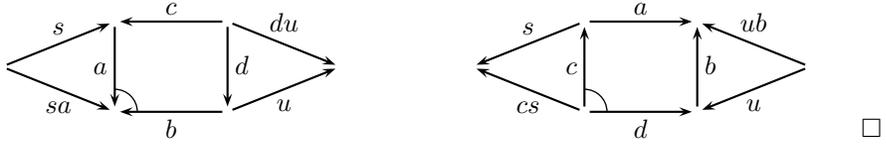
\begin{center}
\begin{picture}(45,7)(0,7)
\psset{nodesep=0.7mm}
\pcline{->}(0,6)(15,0)\tbput{$\ss a$}
\pcline{<-}(15,0)(30,0)\tbput{$b$}
\pcline{->}(30,0)(45,6)\tbput{$\uu $}
\pcline{->}(0,6)(15,12)\taput{$\ss $}
\pcline{<-}(15,12)(30,12)\taput{$c$}
\pcline{->}(30,12)(45,6)\taput{$d\uu $}
\pcline{->}(15,12)(15,0)\tlput{$a$}
\pcline{->}(30,12)(30,0)\trput{$d$}
\psarc[style=thin](15,0){3}{0}{90}
\end{picture}
\HS{10}
\begin{picture}(45,7)(0,7)
\psset{nodesep=0.7mm}
\pcline{<-}(0,6)(15,0)\tbput{$c\ss $}
\pcline{->}(15,0)(30,0)\tbput{$d$}
\pcline{<-}(30,0)(45,6)\tbput{$\uu $}
\pcline{<-}(0,6)(15,12)\taput{$\ss $}
\pcline{->}(15,12)(30,12)\taput{$a$}
\pcline{<-}(30,12)(45,6)\taput{$\uu b$}
\pcline{<-}(15,12)(15,0)\tlput{$c$}
\pcline{<-}(30,12)(30,0)\trput{$b$}
\psarc[style=thin](15,0){3}{0}{90}
\end{picture}
\end{center}
\vspace{4mm}
\end{proof}

Using Lemma~\ref{lem:swap_forms} repeatedly, we deduce that $\RRR$-reduction can be used to go from a fractional normal form to a geodesic expression.

In the sequel, $\One\mm$ stands for $1 \sdots 1$, $\mm$~terms. Note that, for~$\aa, \bb$ in~$\MM$, the multifraction~$\One{\mm}/a/b$ represents the element~$ab^{-1}$ of~$\EG\MM$ when $\mm$ is even, and the element~$a^{-1} b$ when $\mm$ is odd. 

\begin{lemm}\label{lem:to_geo} 
Assume that $\MMt$ is a non-free two-generator Artin--Tits monoid, and $\aa\bb\inv$ and $\cc\inv\dd$ are the fractional normal forms of an element~$\gg$ of~$\EG\MMt$ lying neither in~$\MMt$ nor in~$\MMt\inv$. Let $\vv$ be a geodesic word representing~$\gg$. Then, using $\RRR$-reduction, we can transform each of the multifractions $\One{4|v|}/a/b$ and $\One{4|v|-1}/c/d$ to a multifraction that is represented by $v$.
\end{lemm}

\begin{proof}
The proof is by induction on $|v|$. Since we are assuming that $v$ is not
positive or negative, there is nothing to prove when $|v| \le 1$.

Let $x=\f{v}$ and let $v'$ be the maximal proper suffix of $v$ (of geodesic  length $|v|-1$). By Lemma~\ref{lem:sameletter}, $x$ is equal to either $\f{a}$ or to $\f{c^{-1}}$.  If $x=\f{a}$ then we use Lemma \ref{lem:swap_forms} to transform $\One{4|v|-1}/c/d$ to $\One{4|v|-2}/a/b/1$, and if $x = \f{c^{-1}}$, then we transform $\One{4|v|}/a/b$ to $\One{4|v|-1}/c/d/1$.  In either case, from the fractional normal form for $g$ whose representative word does not start with $x$, we derive a multifraction represented by a word beginning with $x$, and this is equal to the multifraction representing the other fractional normal form of $g$ (or that minus two  initial entries containing $1$). We apply Lemma~\ref{lem:move_to_left} with $t=1$ to move the letter $x$ left to the first or second entry of the multifraction.

In the first case we have identified reductions
\[ \One{4|v|}/a/b \rds \cl{x}/\One{4|v|-1}/a'/b, \quad
 \One{4|v|-1}/c/d \rds \cl{x}/\One{4|v|-3}/a'/b/1, \]
and in the second case we have identified reductions
\[\One{4|v|}/a/b \rds 1/\cl{x}/\One{4|v|-3}/c'/d/1,
\quad
 \One{4|v|-1}/c/d \rds 1/\cl{x}/\One{4|v|-3}/c'/d.
\]

If $a'$ or $c'$ is equal to the identity, then $b$ or $d$ is a positive or
negative element of $\GG'$ and hence by Lemma~\ref{cor:uniquewd}
is represented by the unique geodesic word $v'$, and we are done.
Otherwise, by induction, we can apply sequences of $\RRR$-reductions to transform each of $\One{4|v|-4}/a'/b$ and $\One{4|v|-5}/c'/d$ to
multifractions represented by~$v'$. Hence there are sequences of $\RRR$-reductions that transform each of $\One{4|v|} / a/b$ and $\One{4|v| - 1}/c/d$ to multifractions represented by $v$.
\end{proof}

\subsection{Proof of Proposition~\ref{P:MainProp}}\label{sec:proof}

Using the results of Section~\ref{sec:technical}, we now describe how the steps outlined in Subsection~\ref{sec:strategy} can be executed using $\RRR$-reduction. As in Subsection~\ref{sec:strategy} and for the rest of this section, $\MM$ is a fixed Artin--Tits monoid of sufficiently large type.

\renewcommand\theta{\yy}
\renewcommand\eta{\zz}
\renewcommand\xi{\xx}

Let $w$ be a non-geodesic word representing the multifraction
$\ua = a_1/a_2/\pdots /a_n$ on~$\MM$.
Then $w$ has a leftward length reducing critical factorisation
$w=\alpha w_1 \pdots w_k\beta$. Suppose that the subword $w_i$ starts within
the entry $a_{j_i}$ of the multifraction $\ua$, that is, that we have
\[ a_{j_i}=
 \left \{ \begin{array}{ll} \cl{\theta_i \eta_i} & \text{for $j_i$ odd}\\
 \cl{\eta_i \theta_i} & \text{for $j_i$ even}
\end{array} \right .
\]
for positive words $\eta_i,\theta_i$, where 
$w_i$ shares a prefix with $\eta_i$ or~$\INV\eta_i$. 
Specifically, unless $w_i$ is a positive or negative word, one of
$\eta_i$ or~$\INV\eta_i$ is a proper prefix of
$w_i$, but if $w_i$ is positive or negative, then $w_i$ might instead be a
prefix of $\eta_i$ or~$\INV\eta_i$.

It follows from Lemma \ref{lem:move_to_left} applied with $t=1$ that, via an application of the sequence of moves
$R_{2k,a_1}R_{2k+1,a_2}\pdots R_{2k+j_1-1,a_{j_1-1}}$, we have
$$\One{2k}/a_1/\pdots /a_n\rds \One{2k-2}/a_1/a_2/\pdots /a_{j_1-1}/1 / 1 /a_{j_1}/\pdots /a_n$$ 
and then application of $R_{2k+j_1,\cl{\eta_i}}$  transforms this to
$$\One{2k-2}/a_1/a_2/\pdots /a_{j_1-1}/\cl{\theta_1}/1/\cl{\eta_1}/a_{j_1+1}/ \pdots/ a_n.$$ 
Hence we see that a sequence of such reductions transforms the multifraction  $\One{2k}/a_1/\pdots /a_n$ to a multifraction
$$\ua' := a_1/a_2/\pdots /a_{j_1-1}/\cl{\theta_1}/1/\cl{\eta_1}/a_{j_1+1}/\pdots
/1/\cl{\eta_2}/\pdots\pdots/1/\cl{\eta_k}/\pdots/a_n$$ 
for which $\cl{\eta_i}$ is an entry of~$\ua'$ for each of the subwords $\eta_i$. Let us redefine $j_i$ to denote the index of the entry of $\ua'$ in which $\cl{\eta_i}$ now lies.

Let us assume that $j_k$ is odd, and hence that $w_k$ has a positive prefix -
the other case is similar.

Suppose first that $w_k$ is a positive word.
Then $w_k$ might be distributed over more than one entry of $\ua'$ as a
prefix of a multifraction $\ub := 
\cl{\eta_k}/1/a_{j_k+2}/1/\pdots/a_{j_k+2s}$.
But then a sequence of reductions 
$R_{j_k+2s-1,a_{j_k+2s}}\pdots R_{j_k+1,a{j_k+2}\pdots a_{j_k+2s}}$ would
transform $\ub$ to $\ub':=\cl{\eta_k} a_{j_k+2}\pdots a_{j_k+2s}/\One{2s}$.
So we can assume that $w_k$ lies within a single multifraction entry, as a
prefix of $\eta_k$. The required $\tau^+$-move on $w_k$ can then be executed
just by choosing a different representative word for $w_k$, which does not
change~$\ub'$.

Otherwise, $w_k$ is distributed over more than one entry of $\ua'$, and has a positive prefix. Suppose that $w_k$ occupies $s>2$ such entries. So it has a suffix $\xi_k$ (or~$\INV\xi_k$) which defines a left (resp. right) divisor of the $(j_k+s-1)$-th entry of $\ua'$. Then, by Lemma \ref{lem:swap_forms}, we can apply a reduction $R_{j+k+s-2,\cl{\xi_k}}$, which effectively replaces $w_k$ by a word representing the same group element that occupies occupies at most $s-1$ multifraction entries. A sequence of at most two $s-2$ such reductions transforms $w_k$ into a word occupying two entries, which represents the fractional normal form $ab^{-1}$ of $\cl{w_k}$, with $a$ the $j_k$-th entry and $b$ a right divisor of the $(j_k+1)$-th entry of the new multifraction $\ua''$. (Since $w_k$ is not a positive or negative element of the group, it could not be transformed to a word that occupies fewer than two multifraction entries.) By Lemma~\ref{lem:sameletter} we have $\f{a}=\f{\eta_k}$. Now we apply Lemma~\ref{lem:swap_forms}, with $i = j_k$ and $t=1$, to transform $\ua''$ into a multifraction in which $ab^{-1}$ is represented by its other fractional normal form, $c^{-1}d$, where now $c$ is the $(j_k-1)$-th entry and $d$ the $j_k$-th entry. 

Let $u_k'$ be a product of geodesic words representing $c^{-1}$ and $d$. Then, from Lemma~\ref{lem:letterchange}, we have $\l{c} \ne \f{a} = \f{w_k}$, and so $c^{-1}$ and hence also $u_k'$ begins with the first letter that we wish to append to $w_{k-1}$ to give $w_{k-1}'$. 

We can repeat this process on $w_{k-1}',w_{k-2}' \wdots w_1'$ in turn. Then we can carry out the free reduction between the first letter of $u'_1$ and the final letter of $\alpha$ using a single $\RRR$-reduction. Denote the multifraction after the free reduction by $\ub$. Then a word representing $\ub$ contains each of the words $u_1 \wdots u_k$ as a subword, where each $u_i$ is either a positive or negative word within a single entry of $\ub$, or else it is a word in fractional normal form $ab^{-1}$ or $c^{-1}d$, where the entries of $\ub$ in position $j_i-1$ and $j_i$ are $a$ and $b$, or $c$ and $d$, respectively.

We complete the proof by describing $\RRR$-reductions that transform the multifraction $\One{N}/\ub$ into one that is represented by the word $\alpha'v_1\pdots v_k \beta$, for an appropriately chosen integer $N$. In the case where $u_i$ is a positive or a negative word, $u_i$ and  $v_i$ are geodesic words with $\clp{u_i} = \clp{v_i}$, so no change is needed to the multifraction.

We can apply Lemma~\ref{lem:to_geo} to deal with those $u_i$ that are neither positive nor negative words. We deduce that we can find a sequence of $\RRR$-reductions  that transform $\One{N+2k}/\ua$ via $\One{N}/\ub$ to a multifraction that is represented by the word $\alpha'v_1\pdots v_k \beta$ provided that $N \ge 4 \sum_{i=1}^k |v_i|$. Since clearly both $k$ and $\sum_{i=1}^k |v_i|$ are bounded above by $|w|$, the proof of the Proposition~\ref{P:MainProp} is completed using Lemma~\ref{lem:to_geo}. This completes also the proof of Proposition~\ref{P:Main}. \hfill$\square$

\medskip The results of the current paper leaves the question of whether Artin--Tits monoids of sufficiently large type satisfy Conjecture~$\ConjA$ open, but, at least, we know now that $\RRR$-reduction is relevant for them. In view of this result, and those of~\cite{Dit}, which settle the case of type~FC, the ``first'' case for which nothing is known in terms of reduction (nor of the word problem) is the Artin--Tits monoid with exponents $3, 3, 3, 3, 3, 2$, that is, the monoid\\
\parbox{\linewidth}{\medskip\quad$\MON{\mathtt{a,b,c,d}}{\mathtt{aba=bab, aca=cac, bcb=cbc, ada=dad, bdb=dbd, cd=dc}}$.\medskip}\\
We hope that further progress will arise soon


\end{document}